\documentclass[11pt]{amsart}

\usepackage{times}
\usepackage{amssymb}
\usepackage{graphicx,xspace}
\usepackage{epsfig}
\usepackage{enumitem}
\usepackage{xcolor}
\usepackage{tikz}
\usepackage{gensymb}
\usepackage[T1]{fontenc}
\usepackage[utf8]{inputenc}
\usepackage{bm}

\usepackage{hyperref}

\theoremstyle{plain}

\newtheorem{lemma}{Lemma}[section]
\newtheorem{theorem}[lemma]{Theorem}

\newtheorem{proposition}[lemma]{Proposition}

\theoremstyle{remark}

\newtheorem{definition}[lemma]{Definition}

\newcommand{\C}{\mathbb{C}}

\newcommand{\QQ}{\mathbb{Q}}

\newcommand{\h}{\mathfrak{h}}

\newcommand{\Young}{\mathcal{Y}}

\newcommand{\bfla}{{\bf \lambda}}
\newcommand{\bfw}{{\bf w}}
\newcommand{\trace}{\mathrm{trace}}
\newcommand{\rank}{\mathrm{rank}}

\DeclareMathOperator{\Ch}{Ch}
\DeclareMathOperator{\Cl}{C\ell}

\DeclareMathOperator{\wt}{wt}

\author[V.~Féray]{Valentin Féray}
 \address[V.F.]{LaBRI, Universit\'e Bordeaux 1, 351 cours de la Lib\'eration, 33 400
 Talence, France}
 \email{feray@labri.fr}
 \thanks{VF supported by ANR project PSYCO (ANR-11-JS02-001).}

 \author[I.P.~Goulden]{I. P.~Goulden}
\address[I.G.]{
 Department of Combinatorics \& Optimization,
 University of Waterloo,
 Waterloo, Ontario,
 CANADA N2L 3G1}
 \email{ipgoulden@uwaterloo.ca}
 \thanks{IPG supported by a Discovery Grant from NSERC}



 \title{A multivariate hook formula for labelled trees}

\begin{document}

 \maketitle

 \begin{abstract}
    Several hook summation formulae for binary trees have appeared recently
    in the literature.
    In this paper we present an analogous formula for unordered increasing trees of size
    $r$, which involves $r$ parameters.
    The right-hand side can be written nicely as a product of linear factors.
    We study two specializations of this new formula, including Cayley's enumeration of trees with respect to vertex degree.
    We give three proofs of the hook formula. One of these proofs arises somewhat indirectly, from
    representation theory of the symmetric groups, and in particular uses Kerov's character polynomials. The other proofs are more direct, and
    of independent interest.
 \end{abstract}

 \section{Introduction and the main result}
Hook formulae first appeared in the context of representation theory of the symmetric groups:
 Frame, Robinson and Thrall \cite[Theorem 1]{HookLengthFormula1954} proved 
that the dimension $\chi^{\lambda}( (1^n))$ of the representation associated to a Young diagram $\lambda$ with $n$ boxes,
(which is also the number of increasing labellings of the boxes of $\lambda$)
is given by the simple ratio
\[ \chi^{\lambda}( (1^n))  = \frac{n!}{\prod_{\Box \in \lambda} h(\Box)}, \]
where $h(\Box)$ is the size of the hook attached to the Box $\Box$.

It was subsequently pointed out by D. Knuth \cite[§5.1.4 Exer. 20]{KnuthArtProg3}
that the number $L(T)$ of increasing labellings of the vertices of a rooted tree $T$ can be expressed by using the same kind of formula.
In particular, 
\begin{equation}\label{EqHookOneTree}
    L(T) = \frac{|T|!}{\prod_{v \in T} h_T(v)},
\end{equation}
where $|T|$ is the number of vertices of $T$ and $h_T(v)$ is the size of
the hook $\h_T(v)$ attached to the vertex $v$ in $T$ (see definition below).

At this point, we fix some terminology and notation.
A {\em tree} is an acyclic connected graph.
{\em Rooted} means that we distinguish a vertex; then each edge
can be oriented towards the root and we call respectively {\em father}
and {\em son} the head and tail of the edge.
With this terminology, it is easy to guess what the {\em descendants}
of a vertex are:
they can be defined recursively as the sons and the descendants of the sons.
The hook attached to the vertex $v$ in the tree $T$, denoted by $\h_T(v)$, is the set consisting of $v$ and its descendants.

For another consequence of the rooted tree hook formula \eqref{EqHookOneTree},
recall that there is a well-known one-to-one correspondence between
increasing binary trees with $n$ vertices, and permutations of size $n$,
see {\em e.g.} \cite[p. 23-25]{StanleyEC1}.
Hence, the total number of increasing labellings of all binary trees of size
$n$ is equal to the number of permutations of size $n$, which yields the formula
\begin{equation}
    \sum_{ {\footnotesize \begin{tabular}{c}
        $T$ binary\\
        tree of size $n$
    \end{tabular}}}
    \prod_{v \in T} \frac{1}{h_T(v)}=1.
    \label{EqHookSum}
\end{equation}
Despite their simplicity, both formulae \eqref{EqHookOneTree} and
\eqref{EqHookSum} have been the subject of many research papers.
We mention briefly five directions that these papers have taken:
\begin{itemize}
    \item $q$-analogues of formula \eqref{EqHookOneTree} have been found
        where increasing labellings of a given tree are counted with respect to
        one (or more) statistics: see \cite{BjoernerWachsQHook} and
        \cite[Lemma 5.3]{ChapotonFlorentMarnitonsMoules};
    \item Formula \eqref{EqHookOneTree}
        (and the $q$-analogues mentioned above)
        has been extended to more general
        classes of posets than trees (or forests):
        $d$-complete posets \cite{ProctorDynkin,ProctorMinuscule},
        shrubs \cite[Proposition 3.6]{ChapotonArbustes},
        forests with duplications \cite[Theorem 1.4]{VicEtMoiHook};
    \item In summation formula \eqref{EqHookSum}, the factor $\frac{1}{h_T(v)}$ can
        be replaced by some more complicated function of $h_T(v)$ such that
        the sum over binary trees remains nice.
        An example is the following formula 
        \cite[equation (1.2)]{DuLiuGeneralizedHookSum}
\begin{equation}
    \!\!\!\!\!\!\!\!\!\!\!\! \sum_{ {\footnotesize \begin{tabular}{c}
        $T$ binary\\
        tree of size $n$
    \end{tabular}}}
    \prod_{v \in T} \bigg(x+\frac{1}{h_T(v)} \bigg)=
    \frac{1}{(n+1)!}
    \prod_{i=0}^{n-1}
    \big( (n+1+i)x+n+1-i \big). 
    \label{EqPostnikov}
\end{equation}
    The case $x=0$ of course corresponds to \eqref{EqHookSum},
    the case $x=1$ is due to A. Postnikov 
    \cite[Corollary 17.3]{PostnikovHook}   
    and the general case is due to R. Du
    and F. Liu, who proved a conjecture of A. Lascoux, see
    \cite{DuLiuGeneralizedHookSum} and the references therein.
    Subsequently, G. Han designed an algorithm to discover such equalities,
    finding a generalization of Du and Liu's result, as well as
    many other formulae \cite{HanHookLengthAutomatic};

\item Another direction consists in replacing in summation formula 
    \eqref{EqHookSum} (or in the generalized version \eqref{EqPostnikov})
    binary trees by other families of trees.
    Formulae of this kind for plane forests or $m$-ary trees
    have been given in several papers~\cite{DuLiuGeneralizedHookSum,
    YangGeneralizationHanHookFormula, SunZhangHookFormulaMaryTrees,
    ChenEtAlHookFormulasCombi};

\item Finally, formulae \eqref{EqHookOneTree} and \eqref{EqHookSum} admit a number of
higher level interpretations.
In \cite{HivertNovelliThibonHookEquaDiff}, it is explained how \eqref{EqHookSum}
(and some generalizations) arises from solving differential equations
and can be lifted to the level of combinatorial Hopf algebras.
Probabilistic interpretations of \eqref{EqHookSum} and generalizations
are presented by B. Sagan in \cite{SaganProbabilisticHookFormulas}.
In a different direction, interpretations of
\eqref{EqHookOneTree} and some refinements/generalizations 
have been given in convex geometry~ \cite[Section 6]{VicAdrienAlainEtMoi}
and commutative algebra~\cite{VicEtMoiHook}.
\end{itemize}

In this paper, we follow the third fourth directions above.
Indeed, we present a summation formula,
in which the simple ratio $\frac{1}{h_T(v)}$ is replaced
by a more complicated expression with several parameters.
Besides, we do not work with binary trees,
but instead with {\em unordered increasing rooted trees}:
\begin{itemize}
    \item {\em unordered} means that the sons of a given vertex are not ordered;
    \item {\em increasing} means that the vertices are labelled
        (each integer between $1$ and $r$ is used exactly once)
        and that the label of a son
        is always bigger than the label of its father
        (in particular, the root always gets label $1$).
\end{itemize}

An example of an unordered increasing tree is given in Figure \ref{FigExTree}.
Since the sons of a given vertex are not ordered, we have chosen the convention
of always drawing them in increasing order from left to right.

\begin{figure}[t]
\tikzstyle{vertex}=[circle,draw]
\begin{tikzpicture}
    \node (v1) at (0,0) [vertex] {\footnotesize $1$};
    \node (v2) at (-1,-1) [vertex] {\footnotesize $2$};
    \node (v3) at (-2,-2) [vertex] {\footnotesize $3$};
    \node (v4) at (1,-1) [vertex] {\footnotesize $4$};
    \node (v5) at (-1,-2) [vertex] {\footnotesize $5$};
    \node (v6) at (-1.5,-3) [vertex] {\footnotesize $6$};
    \node (v7) at (1,-2) [vertex] {\footnotesize $7$};
    \node (v8) at (-.5,-3) [vertex] {\footnotesize $8$};
    \node (v9) at (0,-2) [vertex] {\footnotesize $9$};
    \draw (v1) -- (v2);
    \draw (v1) -- (v4);
    \draw (v3) -- (v2);
    \draw (v5) -- (v2);
    \draw (v9) -- (v2);
    \draw (v5) -- (v6);
    \draw (v5) -- (v8);
    \draw (v4) -- (v7);
\end{tikzpicture}
\caption{An increasing unordered tree}
\label{FigExTree}
\end{figure}
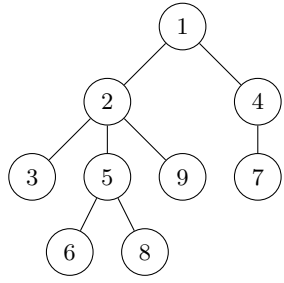

Our summation formula is given in the following theorem, which is the main result of this paper. We use the notation for falling factorials $(a)_m =a(a-1)\cdots (a-m+1)$ for positive integers $m$, with $(a)_0=1$, and $(a)_m=1/(a-m)_{-m}$ for negative integers $m$.

\begin{theorem}
Let $r\geq 1$ be an integer and $k_1, \cdots,k_r$ be formal variables,
    with $K=\sum_{i=1}^r k_i$.
    For an unordered increasing tree $T$ with $r$ vertices,
    define the weight to be
    \[ \wt(T) = \prod_{v=2}^r k_{f(v)}
            \bigg( \Big(\sum_{u \in \h_T(v)} k_u \Big) - h_T(v) + 1 \bigg),\]
    where $f(v)$ stands for the father of $v$ in $T$.
    Then
    \begin{equation}\label{hookform}
    \sum_T \wt(T)
    = k_1 \cdots k_r
    (K-1)_{r-2},
    \end{equation}
    \label{ThmHookFormula}
    where the sum runs over all unordered increasing trees on $r$ vertices.
\end{theorem}
For example, the weight of the tree given in Figure \ref{FigExTree} is
\begin{multline*}
    k_1 (k_2+k_3+k_5+k_6+k_8+k_9-5) \cdot k_2 k_3 \cdot k_1 (k_4+k_7-1)
\\ \cdot k_2 (k_5 +k_6 +k_8-2) \cdot k_5 k_6
\cdot k_4 k_7 \cdot k_5 k_8 \cdot k_2 k_9.
\end{multline*}
Note that, if $v$ is a leaf, its contribution to the weight is $k_{f(v)} k_v$.
Since each vertex is either a leaf or the father of another vertex,
the quantity $\wt(T)$ is always divisible by $k_1 \cdots k_r$ 
(except for $r=1$).

We refer to~\eqref{hookform} as our \emph{hook formula}.
We point out the fact that the formula for trees of size $r$
involves $r$ independent parameters,
while formula \eqref{EqPostnikov} and all formulae in 
\cite{HanHookLengthAutomatic} involve a fixed number
of parameters.
As mentioned above, for $r>1$, the monomial $k_1 \cdots k_r$
divides all terms of the sum,
but the latter do not share any other factors.
Thus it is quite remarkable that the right-hand side, which is a polynomial
in $r$ parameters, can be written as a product of simple linear factors. (Note that in the case $r=1$, we have $(K-1)_{r-2}=k_1^{-1}$, which cancels the factor $k_1$.)

In Section \ref{SectSpe} we present two specializations of our result:
an analogue
of the aforementioned hook formula of Postnikov, and the multivariate enumeration of Cayley trees with respect to vertex degree.
In our opinion, this makes Theorem~\ref{ThmHookFormula} interesting in itself.

Another interesting feature of this new hook formula is the connection with representation theory of the symmetric group. This link is explained in Section~\ref{SectRepresentation}, where we give our first proof of Theorem~\ref{ThmHookFormula}. This proof uses Kerov's character polynomials, and does not seem related to the Frame-Robinson-Thrall formula. The proof is quite involved, and reasonably indirect, so we also give two inductive proofs of the hook formula that are more direct. The first of these direct proofs, given in Section~\ref{SectProof1}, uses elementary operators on polynomials. The second of these direct proofs is given in Section~\ref{SectProof2}, and uses Lagrange's Implicit Function Theorem in many variables.


\section{Two specializations of the hook formula}
\label{SectSpe}
\subsection{An analogue of Postnikov's formula}
Here we consider the specialization of all variables $k_1,\dots,k_r$
to the same value $k$.
Then the weight of an unordered increasing tree $T$ in Theorem~\ref{ThmHookFormula} becomes
\begin{align*}
    \wt'(T) = \wt(T) \bigg|_{k_i=k} &= k^{r-1} \prod_{v=2}^r \bigg( (k-1) h_T(v) +1 \bigg) \\
&= \frac{k^{r-1}}{(k-1)r+1} \prod_{v \in T} \bigg( (k-1) h_T(v) +1 \bigg) .
\end{align*}
Therefore, setting $x=k-1$, our hook formula becomes
\begin{equation}
    \sum_{ {\footnotesize \begin{tabular}{c}
        $T$ increasing\\
        unordered tree\\
        of size $r$
    \end{tabular}}}
    \prod_{v \in T} (x h_T(v) + 1 )= (x+1)
    \prod_{i=1}^{r-1} (x\cdot r +i)
    \label{EqHookAllParEqual}
\end{equation}
Using the fact (equation~\eqref{EqHookOneTree}) that there are 
$n!/(\prod_{v\in T} h_v(T))$
increasing labellings for each binary tree $T$,
equation \eqref{EqPostnikov} can be rewritten as 
\begin{equation}
    \!\!\!\!\!\!\!\!\!\!\!\!  \sum_{ {\footnotesize \begin{tabular}{c}
        $T$ increasing\\
        binary tree\\
        of size $n$
    \end{tabular}}}
    \prod_{v \in T} \left(x h_T(v)+1 \right)=
    \frac{1}{n+1}
    \prod_{i=0}^{n-1}
    \big( (n+1+i)x+n+1-i \big).
    \label{EqPostnikovRewritten}
\end{equation}
Thus the specialization with equal parameters of our formula is an analogue
of Postnikov's formula for another family of trees.
Unfortunately, a short computer exploration suggests that 
equation~\eqref{EqPostnikovRewritten} does not seem to have such a nice multivariate 
refinement as Theorem~\ref{ThmHookFormula}.

\subsection{Multivariate enumeration of Cayley trees}
By definition, a Cayley tree is a 
tree\footnote{Cayley trees are not embedded
in the plane and have no root,
they are only specified by an adjacency matrix.}
with distinguishable vertices.
As early as 1860 \cite{BorchardtCayley},
C.W. Bochardt proved that the number of trees
with vertex set $[r] = \{ 1, \cdots ,r \}$ is $r^{r-2}$.
As noticed by A. Cayley \cite{CayleyTrees}, his proof also leads to
the following multivariate enumeration formula
for what are now called Cayley trees:
\begin{equation}
    \label{EqCayleyRefined}
    \sum_{\footnotesize \begin{tabular}{c}
    $U$ Cayley tree\\
    with vertex set $[r]$
\end{tabular}}
\!\!\!\!\!\! k_1^{d_1(U)} \cdots k_r^{d_r(U)} =
k_1 \cdots k_r K^{r-2},
\end{equation}
where $d_i(U)$ denotes the degree of the vertex $i$ in a tree $U$.

We will show that the specialization $k_1,\dots,k_r \to \infty$, that is
the highest degree term in $k$ of our hook formula, corresponds to
\eqref{EqCayleyRefined}.
Hence our hook formula can be viewed as a non-homogeneous extension of
the multivariate enumeration of Cayley trees.

To do this, we define a mapping $\varphi$ from Cayley trees
with vertex set $V$ to increasing unordered trees with label set $V$,
where $V$ is a finite nonempty set of positive integers. Consider a Cayley tree $U$ with vertex set $V$.
The definition is inductive and produces an increasing unordered tree $T=\varphi (U)$ as follows:
\begin{itemize}
    \item Let $\ell=\min V$. If $|V|=1$, then $T$ has a single vertex, with label $\ell$. Otherwise, remove vertex $\ell$ and all incident edges from $U$, to obtain
        a forest whose connected components are Cayley trees
        $U_1, U_2,\dots$ ;
    \item Apply $\varphi$ inductively to $U_1, U_2,\dots$;
    \item Take the disjoint union of all $T_i=\varphi(U_i)$, and add a vertex (which is the root vertex of $T$) with 
        label $\ell$, joined to the root vertices of all $T_i$.
\end{itemize}
The mapping $\varphi$ is clearly not injective in general.
If $T$ is an increasing unordered tree with label set $V$, then
the elements $U$ of the preimage $\varphi^{-1}(T)$
can be obtained inductively as follows:
\begin{itemize}                              
    \item Let $\ell=\min V$. If $|V|=1$, then $U$ has the single vertex $\ell$. Otherwise, remove the root vertex of $T$ (which has label $\ell$), to obtain
        the increasing unordered trees $T_1, T_2,\dots$;
    \item Select an element $U_i$ in each set $\varphi^{-1}(T_i)$;
    \item Take the disjoint union of all $U_i$, choose one vertex in each $U_i$
        and add a vertex with label $\ell$ joined to all selected vertices.
\end{itemize}
For a given increasing unordered tree $T$, denote
\[ \wt''(T) = \sum_{U: \varphi(U)=T} \; \prod_{v \in V} k_v^{d_v(U)}.\]
The above description of $\varphi^{-1}(T)$ implies that
\[ \wt''(T) = \prod_{T_i} \wt''(T_i) 
\bigg( k_{\ell}  \sum_{v \in T_i} k_v \bigg),\]
where $\ell$ is the label of the root and 
the product is taken over the trees $T_1, T_2,\dots$ obtained
by removing the root of $T$.
An immediate induction yields
\[ \wt''(T) = \prod_{v=2}^r k_{f(v)}
            \bigg( \sum_{u \in \h_T(v)} k_u \bigg),\]
with the same notation as in Theorem~\ref{ThmHookFormula}.
We observe that $\wt''(T)$ is exactly the highest degree term in 
$\wt(T)$ and therefore,
as an immediate corollary of Theorem~\ref{ThmHookFormula}, we get
\[\sum_{ {\footnotesize \begin{tabular}{c}
        $T$ increasing\\
        unordered tree\\
        of size $r$
    \end{tabular}}} \!\!\!\!\!\! \wt''(T)
    = k_1 \cdots k_r K^{r-2}, \] 
which is the multivariate enumeration formula~\eqref{EqCayleyRefined} for Cayley trees.

\section{Kerov character polynomials}
\label{SectRepresentation}

In this section, we explain how Theorem~\ref{ThmHookFormula} arises
from computations in representation theory of the symmetric group.
In fact, the two sides of our hook formula correspond to the same coefficient
of the so called {\em Kerov character polynomials},
computed in two different ways.

In paragraph \ref{SubsectDefKerov}, we explain Kerov character polynomials
and which coefficient we want to compute.
Then, in paragraphs~\ref{SubsectKerovComb}, \ref{SubsectKerovMacdo}
and \ref{SubsectKerovFrob}, we give different ways to compute this coefficient,
which lead to our hook formula.
The first two approaches lead to the same result, but we have chosen to
present both to be more comprehensive on the subject.

\subsection{Definitions}\label{SubsectDefKerov}

Let us consider, for each $n$, the family of symmetric groups $S_n$.
It is well-known (see, {\em e.g.}, \cite[Chapter 2]{Sagan}) that
both conjugacy classes and irreducible representations of $S_n$ can
be indexed canonically by partitions of $n$,
so the character table of $S_n$ is a collection of numbers
$\chi^\lambda(\mu)$, where $\lambda$ and $\mu$ run over partitions of $n$
and are, respectively, the indices of the irreducible representation
and the conjugacy class.

Following S. Kerov and G. Olshanski \cite{KerovOlshanski1994},
for any partition $\mu$ of size $k$,
we shall consider the function $\Ch_\mu$
on the set $\Young$ of all Young diagrams (or equivalently
of all partitions of all sizes) defined by:
\[\Ch_\mu(\lambda) = \begin{cases}
    0  &\text{if } n< k;\\
    n(n-1)\ldots (n-k+1) \frac{\chi^\lambda\big(\mu \cup (1^{n-k}) \big)}
    {\chi^\lambda(Id_n)} &\text{otherwise,}
\end{cases}\]
where $n$ is the size of $\lambda$.

We also consider another family of functions on Young diagrams:
the free cumulants $(R_k)_{k \geq 2}$ of the transition measure
(for their definition we refer to \cite[Section 1]{BianeAsymptoticsCharacters}).
It has been shown by S. Kerov \cite[Theorem 1]{Biane2003}
(the reference given deals only with the case of a one-part
partition $\mu$, but the proof can be readily extended to the general case)
that there
exist polynomials $K_\mu$ such that, as functions on all Young diagrams,
\begin{equation}\label{EqDefKerovPoly}
    \Ch_\mu=K_\mu(R_2,R_3,\dots). 
\end{equation}
These polynomials are called Kerov character polynomials.
Their coefficients have been the subject of many research articles
in the last few years, see \cite{DolegaFeraySniady2008} and references therein.
Here we focus on the coefficient of a single $R_j$ ({\em linear} coefficient) for
the maximal value of $j$, that is
\[j=|\mu|-\ell(\mu)+2.\]
This coefficient has a very compact expression that we prove in the next paragraph (we use throughout the notation $[A]B$ to denote the \emph{coefficient} of $A$ in the expansion of $B$).
\begin{proposition}
    \label{PropTopDeg2OnePart}
Let $\mu$ be a partition and $j=|\mu|-\ell(\mu)+2$.
Then
\[ [R_j] K_\mu = (-1)^{\ell(\mu)-1} \left( \prod_{i=1}^{\ell(\mu)} \mu_i \right)
\frac{(|\mu|-1)!}{(|\mu|-\ell(\mu)+1)!}.\]
\end{proposition}



\subsection{Combinatorial interpretation of Kerov polynomials}
\label{SubsectKerovComb}

Linear coefficients in Kerov polynomials have a quite simple combinatorial
interpretation, established by P. Biane~\cite[Theorem 5.1]{Biane2003} for one-part partitions $\mu$, and by A. Rattan and P. \'Sniady~\cite[Theorem 19]{RattanSniady2008} for arbitrary partitions $\mu$:

$(-1)^{\ell(\mu)-1} [R_j] K_\mu$ is the number of pairs $(\sigma_1,\sigma_2)$ such that
\begin{itemize}
    \item $\sigma_1$ and $\sigma_2$ are permutations in $S_{|\mu|}$ with
        \begin{equation}\label{EqFacto}
            \sigma_1 \sigma_2 = \sigma_\mu,
        \end{equation}
        where $\sigma_\mu=(1 \dots \mu_1) (\mu_1 +1 \ \dots \mu_2) \dots$;
    \item $\sigma_2$ is a long cycle;
    \item $\sigma_1$ has $j - 1$ cycles.
\end{itemize}

Note that the absolute lengths\footnote{The absolute length of a permutation
    is the minimal number of factors needed to write it as a product of transpositions.
    It should note be confused with its Coxeter length.}
of $\sigma_1$ and $\sigma_\mu$ are $|\mu|-(j-1)=\ell(\mu)-1$ and $|\mu|-\ell(\mu)$.
These two numbers sum up to $|\mu|-1$.
This allows to use a theorem of F.~Bédard and A.~Goupil,
who counted the number of factorizations~\eqref{EqFacto}
where $\sigma_1$ has a given cycle-type $\lambda$ (here, $|\lambda|=|\mu|$
and $\ell(\lambda)=j-1$).
They obtained the following number \cite[Theorem 3.1]{BedardGoupilFactorizations}
(see also \cite[Theorem 2.2]{GouldenJackson1992}):
\[ \frac{(\ell(\mu)-1)! (j-2)! \prod_i \mu_i}{m_1(\lambda)! m_2(\lambda)! \cdots},\]
where $m_i(\lambda)$ is the number of parts of $\lambda$ equal to $i$, $i\geq 1$. To obtain $[R_j] K_\mu$, we have to sum over all possible cycle-types $\lambda$:
\[ (-1)^{\ell(\mu)-1} [R_j] K_\mu = \frac{(\ell(\mu)-1)!}{j-1} \prod_i \mu_i
    \!\!\!\!\!\! \sum_{\substack{\lambda \vdash |\mu|,\\\ell(\lambda)=|\mu|-\ell(\mu)+1}}
    \!\!\!\!\!\! \frac{(j-1)!}{m_1(\lambda)! m_2(\lambda)! \cdots}
    \]
The term indexed by $\lambda$ in the sum counts the number of sequences
$i_1,\dots,i_{j-1}$ that are permutations of $\lambda$.
Hence the sum is the number of sequences $i_1,\dots,i_{j-1}$ of positive integers
of sum $|\mu|$, that is $\binom{|\mu|-1}{j-2}$.
It is then straightforward to see that the expression above simplifies to the
one in Proposition~\ref{PropTopDeg2OnePart}.

\subsection{Macdonald symmetric functions}\label{SubsectKerovMacdo}

In this paragraph, we present another approach to Proposition~\ref{PropTopDeg2OnePart},
which relies on a basis of the symmetric function ring introduced
by I.G. Macdonald.

Consider the center $Z(\C[S_n])$  of the symmetric group algebra of size $n$.
A basis is given by the conjugacy class sums, that is
\[\Cl_\lambda = \sum_{\text{cycle-type}(\sigma)=\lambda} \sigma.\]
Since $Z(\C[S_n])$ is an algebra, there exist constants $c^{\lambda}_{\mu,\nu}$ such that,
for any two partitions $\mu$ and $\nu$ of size $n$,
\[\Cl_\mu \Cl_\nu = \sum_{\lambda \vdash n} c^{\lambda}_{\mu,\nu} \Cl_\lambda.\]
These constants are called {\em structure constants} or {\em connection coefficients}
of $Z(\C[S_n])$ and have been widely studied in the literature.

Macdonald~\cite[Exercises I.7.24, I.7.25]{McDo} gave an explicit construction of
a basis $u_\lambda$ of the symmetric function ring,
which can be characterized as follows:
\begin{itemize}
    \item $u_\lambda$ is homogeneous of degree $|\lambda|$;
    \item if $\lambda$ has only one part, then $u_\lambda$ is given by
        \[ u_{(n)} = -p_n, \]
        where $p_n$ is the $n$-th power sum;
    \item for a partition $\lambda$,
        denote $\bar{\lambda}$ the partition obtained from $\lambda$ by adding one
        to every part.
         Then, for any partitions $\mu$, $\nu$ and $n \geq |\bar{\mu}| +  |\bar{\nu}|$,
        \begin{equation}\label{EqUStruct}
            u_\mu u_\nu = \sum_{\lambda \vdash |\mu|+|\nu|} \
            c_{\bar{\mu} 1^{n-|\bar{\mu}|},\bar{\nu} 1^{n-|\bar{\nu}|}}^{
            \bar{\lambda} 1^{n-|\bar{\lambda}|}} u_\lambda
    \end{equation}
        where $c$ is the structure constant of the center of the symmetric group algebra
        defined above.
\end{itemize}
This construction can be found
in paper \cite{GouldenJacksonMacdonaldSym}
(see in particular Theorem 3.2 and Proposition 4.1,
which corresponds to the properties above).

Note that it is well-known~\cite[Lemma 3.9]{FaharatHigman1959} that the coefficients in the
right-hand side of \eqref{EqUStruct}
do not depend on $n$ (because $|\lambda|=|\mu|+|\nu|$).

We will see that Kerov polynomials contain in some sense Macdonald symmetric functions.
To do this, consider, as in \cite{DolegaFerayKerovJack} the gradation $\deg_2$
on the algebra $\Lambda$ generated by $R_k$ (for $k \ge 2$) defined
\[\deg_2(R_k)=k-2.\]
One can show that free cumulants are algebraically independent
so the definition makes sense.
Then, one has the following properties:
\begin{itemize}
    \item The top component of $K_k$ is $R_{k+1}$.
        Indeed consider a monomial $\prod_{i=1}^t R_{j_i}$ appearing to the top
        component of $K_k$ for $\deg_2$, \emph{i.e.} such that
        \[ \sum_{i=1}^t (j_i-2) = k-1. \]
        Then we must also have $\sum j_i \leq k+1$ \cite[Section 6]{Biane2003}.
        These two equations imply $t \leq 1$, which means that
        only $R_{k+1}$ appears in the top component of $K_k$
        (and its coefficient is known to be $1$);
    \item Let $\mu$ and $\nu$ be two partitions. Then
        one has
        \begin{multline*}
            \frac{K_{\bar{\mu}}}{z_{\bar{\mu}}} \cdot
        \frac{K_{\bar{\nu}}}{z_{\bar{\nu}}} =
        \sum_{\lambda \vdash |\mu|+|\nu|}
         c_{\bar{\mu} 1^{n-|\bar{\mu}|},\bar{\nu} 1^{n-|\bar{\nu}|}}^{
                    \bar{\lambda} 1^{n-|\bar{\lambda}|}} 
    \frac{K_{\bar{\lambda}}}{z_{\bar{\lambda}}}
    + \text{smaller degree terms for }\deg_2,
\end{multline*}
where $z_\pi$ is the classical constant $\prod_i i^{m_i} m_i!$
if $\pi$ is written as $1^{m_1} 2^{m_2} \cdots$
in exponential notation \cite[Chapter 1]{McDo}.
This second property can be deduced from \cite[Proposition 4.5]{IvanovOlshanski2002}:
we skip details here.
    \end{itemize}

Consider the algebra isomorphism between the subalgebra $\QQ[R_3,R_4,\dots]$
of $\Lambda$ and the symmetric function ring sending $R_{j+2}$ to $-(j+1) p_j$.
Then the top component of
 $\frac{K_{\bar{\lambda}}}{z_{\bar{\lambda}}} $
is sent to $u_\lambda$ because of the two properties above.

Hence, this top component can be computed using results on $u_\lambda$,
in particular \cite[Lemmas 7.1 and 7.2]{GouldenJacksonMacdonaldSym}.
If $j -2 = |\bar{\nu}|-\ell(\bar{\nu})=|\nu|$, then
\begin{align*}
    [R_j] K_{\bar{\nu}} =  \frac{-z_{\bar{\nu}}}{j-1} [p_{j-2}] u_\nu
&= \frac{-z_{\bar{\nu}}}{(j-1)(j-2)} [h_\nu] [s^{j-2}]
\frac{1}{\left(\sum_{m \geq 0} h_m s^m\right)^{j-2}} \\
&= \frac{-z_{\bar{\nu}}}{(j-1)(j-2)} \binom{-(j-2)}{m_1(\nu),m_2(\nu),\dots} \\
&= \frac{-z_{\bar{\nu}}}{(j-1)(j-2)} (-1)^{\ell(\nu)} \binom{j-2+ \ell(\nu)-1}{m_1(\nu),m_2(\nu),\dots}
\end{align*}
Simplifying the expression above and setting $\mu=\bar{\nu}$, we obtain
Proposition~\ref{PropTopDeg2OnePart}.

\subsection{Using the generalized Frobenius formula}\label{SubsectKerovFrob}
The most efficient way to compute the polynomials $K_\mu$
with a computer is to use the generalized Frobenius
formula \cite[Theorem 5]{RattanSniady2008}.
To state it, we need the notion of boolean cumulants $B_k$ (for $k \ge 2$)
of the transition measure.
They are functions on the set of all Young diagrams and they form
another algebraic basis of $\Lambda$ such that   
\[B_k=R_k + \text{non-linear terms}.\]
This implies that $[B_k]\Ch_\mu = [R_k]\Ch_\mu$, which is by definition $[R_k]K_\mu$
(see equation~\eqref{EqDefKerovPoly}).
Lastly, we denote by $H(z)$ the generating function of boolean cumulants
(which has coefficients in the ring $\Lambda$):
\[H(z) = z - B_2 z^{-1} - B_3 z^{-2} - \cdots.\]
The following result of A. Rattan and P. \'Sniady expresses
the normalized character
values $\Ch_\mu$ in terms of boolean cumulants:

\begin{theorem}[\cite{RattanSniady2008}]
    For any integers $\mu_1 \ge \cdots \ge \mu_r \geq 1$,
    \begin{multline} \label{EqGenFrob}
        (-1)^r \mu_1 \cdots \mu_r \Ch_{\mu_1,\dots,\mu_r} \\
        =[z_1^{-1}] \cdots [z_r^{-1}] \bigg[ \left(
        \prod_{1 \leq u \leq r} H(z_u) H(z_u-1) \cdots H(z_u-\mu_u+1) \right)\\
        \times
        \prod_{1 \leq s < t \leq r} \frac{(z_s-z_t)(z_s-z_t+\mu_t-\mu_s)}{(z_s-z_t-\mu_s)(z_s-z_t+\mu_t)}.
        \bigg]
    \end{multline}
    The right-hand side of \eqref{EqGenFrob} should be understood as follows:
    we expand the expression appearing there as a power series in decreasing powers
    of $z_r$ with coefficients being $\Lambda$-valued functions
    of $z_1 ,\dots , z_{r-1}$ and select the appropriate coefficient.
    We repeat this procedure with respect to $z_{r-1} , z_{r-2} , \dots , z_1$.
    \label{ThmGenFrob}
\end{theorem}
In Proposition \ref{PropTopDeg2OnePart}, we are interested 
in the coefficient of a single $R_j$ of maximal degree.
As mentioned above,
it is equivalent to look at the coefficient of a single $B_j$ of maximal degree.
In this paragraph, we try to understand this coefficient using 
Theorem~\ref{ThmGenFrob}.

Let us first see what happens in the case $r=2$:
we consider the coefficient of $B_{\mu_1+\mu_2}$ in $\Ch_{\mu_1,\mu_2}$.
The right-hand side of \eqref{EqGenFrob} can then be written as
\begin{multline}\label{EqGenFrob2}
    [z_1^{-1}] H(z_1) \cdots H(z_1 -\mu_1+1) \\
    [z_2^{-1}] H(z_2) \cdots H(z_2 -\mu_2+1)
    \frac{(z_1-z_2)(z_1-z_2+\mu_2-\mu_1)}{(z_1-z_2-\mu_1)(z_1-z_2+\mu_1)}
\end{multline}
When we expand the fraction in decreasing powers of $z_2$, no positive powers appear.
In a factor $H$, the maximal exponent of $z_2$ is $1$.
Hence, the term $B_h z_2^{-(h-1)}$ for $h \geq \mu_2+2$ will not contribute to the
coefficient in $z_2^{-1}$.
In particular, one can not obtain $B_{\mu_1+\mu_2}$, which is what we are looking for.
Therefore each term $H(z_2-c)$ can be replaced by $z_2-c$.

That being said, to obtain at the end the $B_j$ of maximal index,
we have to keep the biggest possible power of $z_1$ in the coefficient of $z_2^{-1}$.
To do that, we notice, that if we consider the total degree in the $z$-variable set
\begin{align*}
        z_2 - c &= z_2 + \text{smaller degree terms};\\
            \frac{(z_1-z_2)(z_1-z_2+\mu_t-\mu_s)}{(z_1-z_2-\mu_s)(z_1-z_2+\mu_t)} &=
                1 + \frac{\mu_2 \mu_1/z_2^2}{(1-z_1/z_2)^2} + \text{smaller degree terms}.
\end{align*}
Hence we have
\begin{align*}
    [z_2^{-1}] H(z_2) &\cdots  H(z_2 -\mu_2+1)
        \frac{(z_1-z_2)(z_1-z_2+\mu_2-\mu_1)}{(z_1-z_2-\mu_1)(z_1-z_2+\mu_1)} \\
         &= [z_2^{-1}] \left( z_2^{\mu_2} \cdot \frac{\mu_2 \mu_1/z_2^2}{(1-z_1/z_2)^2} \right)
        + \text{smaller degree terms in }z_1 \\
        &= \mu_1 \mu_2^2 z_1^{\mu_2-1} + o(z_1^{\mu_2-1}).
\end{align*}
Plugging this into equation~\eqref{EqGenFrob2} and setting all $B_j$ to $0$,
except $B_{\mu_1+\mu_2}$, we obtain
\begin{multline*}
    [B_{\mu_1+\mu_2}] \mu_1 \mu_2 \Ch_{\mu_1,\mu_2}
    = [B_{\mu_1+\mu_2}] [z_1^{-1}] \\
\times\prod_{i=0}^{\mu_1-1} \left( z_1 - i - B_{\mu_1+\mu_2} (z_1-i)^{-(\mu_1+\mu_2-1)} \right)
\big( \mu_1 \mu_2^2 z_1^{\mu_2-1} + o(z_1^{\mu_2-1}) \big) .
\end{multline*}
When we expand the product on the right-hand side,
the term containing $B_{\mu_1+\mu_2}$ of maximal degree in $z_1$ is obtained
by picking $\mu_1-1$ factors $z_1$, one factor $-B_{\mu_1+\mu_2} z_1^{-(\mu_1+\mu_2-1)}$
and finally the factor $\mu_1 \mu_2^2 z_1^{\mu_2-1}$ in the last parenthesis.
We have $\mu_1$ ways to do so (corresponding to the choice of the index $i$
from which we take the term $B_{\mu_1+\mu_2} z_1^{\mu_1+\mu_2-1}$) and thus
\begin{multline*}
    [B_{\mu_1+\mu_2}] \mu_1 \mu_2 \Ch_{\mu_1,\mu_2} \\
    = [z_1^{-1}] \bigg( - \mu_1 z_1^{\mu_1-1} z_1^{\mu_1+\mu_2-1}
(\mu_1 \mu_2^2 z_1^{\mu_2-1}) 
 + \text{smaller degree terms in }z_1 \bigg)
= - \mu_1^2 \mu_2^2
\end{multline*}
Since $[B_{\mu_1+\mu_2}] \Ch_{\mu_1,\mu_2}=[R_{\mu_1+\mu_2}] \Ch_{\mu_1,\mu_2}$,
we recover Proposition~\ref{PropTopDeg2OnePart}
in the case $\ell(\mu)=2$.

Let us consider now the general case.
We want to compute the coefficient of $B_j$ in $\Ch_{\mu_1,\dots,\mu_r}$
for $j -2=\sum_i (\mu_i -1)=K-r$.
As in the case $\ell(\mu)=2$,
when we extract the coefficient of some $z_t$ (for $t >1$),
we have to keep only the highest degree term in the $z$-variable set.
Therefore, for a fixed index $t>1$, we can replace $H(z_t-c)$ by $z_t$
and use the approximation
\begin{multline}\label{EqApproxFraction}
    \prod_{1\leq s < t} \frac{(z_s-z_t)(z_s-z_t+\mu_t-\mu_s)}{(z_s-z_t-\mu_s)(z_s-z_t+\mu_t)} \\
= 1 + \sum_{1\leq s < t} \frac{\mu_t \mu_s/z_t^2}{(1-z_s/z_t)^2} + \text{smaller degree terms}.
\end{multline}
So the highest degree term in $z_1$ after successive extractions
of the coefficients of
$z_r^{-1}, z_{r-1}^{-1} \cdots z_2^{-1}$ is
\[ [z_2^{-1}] \cdots [z_r^{-1}]
\left( \prod_{t=2}^r z_t^{\mu_t} \left[1 + \sum_{1\leq s < t}
\frac{\mu_t \mu_s/z_t^2}{(1-z_s/z_t)^2}) \right] \right)\]
Exchanging the product and summation symbol,
we get a sum over the following set:
for each $t>1$, we have to choose
an integer $s<t$ (we can not choose the summand $1$ in the bracket,
because we would get $z_t$ with a positive power, 
while we want to extract the coefficient of $z_t^{-1}$).
These choices can be represented as an unordered increasing
tree $T$ with $r$ vertices, in which $s$ in the father of $t$.
In the case $r=2$, we only had one summand.

If $f(t)$ denotes the father of $t$ in a tree $T$,
the summand associated to $T$ is
\begin{equation}
    \label{EqTech1}
    A_T := [z_2^{-1}] \cdots [z_r^{-1}] \left( \prod_{t=2}^r z_t^{\mu_t} 
\frac{\mu_t \mu_{f(t)}/z_t^2}{(1-z_{f(t)}/z_t)^2} \right).
\end{equation}
We then use the expansion
\[ \frac{1}{(1-z_{f(t)}/z_t)^2} = \sum_{m_t \geq 1}
m_t (z_{f(t)}/z_t)^{m_t-1}\]
and rewrite equation~\eqref{EqTech1} as
\begin{equation}
    \label{EqTech2}
    A_T = [z_2^{-1}] \cdots [z_r^{-1}] \left( \prod_{t=2}^r z_t^{\mu_t} 
    \mu_t \mu_{f(t)} z_t^{-2} \sum_{m_t \geq 1} m_t (z_{f(t)}/z_t)^{m_t-1} \right).
\end{equation}
A straightforward induction beginning at the leaves of $T$ and going up to
the root shows that the coefficients of $z_2^{-1} \cdots z_r^{-1}$ 
corresponds to the summand
\[ m_t= \sum_{u \in \h_T(t)} \mu_u - h_T(t) + 1, \]
where $h_T(t)=|\h_T(t)|$,
and $\h_T(t)$ is the hook of $t$, as defined in the introduction.
So, finally equation~\eqref{EqTech2} reduces to
\[ A_T = z_1^{K-\mu_1+r-1} \prod_{t=2}^r \mu_t \mu_{f(t)} 
\left( \sum_{u \in \h_T(t)} \mu_u - h_T(t) + 1 \right). \]
Coming back to formula~\eqref{EqGenFrob},
the coefficient $[B_j]\Ch_{\mu_1,\dots,\mu_r}$ is given by
\begin{multline*}
    [B_{K-r+2}] (-1)^r \mu_1 \cdots \mu_r \Ch_{\mu_1,\dots,\mu_r}\\
    = [B_{K-r+2}][z_1^{-1}] H(z_1) \cdots H(z_1-\mu_1+1) 
    \left(\sum_T A_T\right).
\end{multline*}
As in the case $r=2$, the extraction of the coefficient of $B_{K-r+2} z_1^{-1}$
yields an extra factor $\mu_1$ and the equation above simplifies to
\[ (-1)^{r-1} [B_j] \Ch_{\mu_1,\dots,\mu_r} = \sum_T \bigg( \prod_{t=2}^r \mu_{f(t)}
\bigg( \sum_{u \in \h_T(t)} \mu_u - h_T(t) + 1 \bigg) \bigg) . \]
Together with Proposition~\ref{PropTopDeg2OnePart} and the remark above that
\[ [B_j] \Ch_{\mu_1,\dots,\mu_r} = [R_j] \Ch_{\mu_1,\dots,\mu_r}
= [R_j] K_{\mu_1,\dots,\mu_r}, \]
this proves (in a very indirect way) Theorem \ref{ThmHookFormula}.

\section{Elementary operators on polynomials}
\label{SectProof1}

The purpose of this section is to give the first of our two direct proofs of the hook formula (Theorem \ref{ThmHookFormula}),
which uses operators on polynomials. We proceed by induction on $r$, with base case $r=1$, for which the theorem is trivially true.

In the induction,
we will consider trees whose label sets are not necessarily an interval
$[r]=\{ 1,\dots ,r\}$.
Thus we use the notation $X(T)$ for the label set of a tree $T$.
We shall use the following construction on trees.
\begin{definition}
    Let $T_1$ and $T_2$ be two unordered increasing trees with disjoint
    sets of labels.
    Assume that the label of the root of $T_1$ is smaller than the
    label of the root of $T_2$.
    Then, we can construct a new unordered increasing tree,
    called {\em grafting of $T_2$ on $T_1$}, denoted $T_2 \bullet T_1$,
    defined as follows:
    \begin{itemize}
        \item its set of labels is $X(T_1) \sqcup X(T_2)$ ;
        \item its root label is the root label of $T_1$ ;
        \item the vertex with the root label of $T_2$ is a son of the root ;
        \item every non-root vertex of $T_1$ (resp. $T_2$)
            has the same father in $T_2 \bullet T_1$ as in $T_1$ (resp. $T_2$).
    \end{itemize}
\end{definition}
This construction is illustrated on Figure \ref{FigGrafting}.

Now consider an arbitrary (unordered increasing) tree $T$ of size $r>1$.
The vertices labelled $1$ and $2$ must be joined by an edge because $T$ is increasing, 
so $T$ can be obtained in a unique way
by grafting a tree $T_2$ with root $2$ on a tree $T_1$ with root $1$.

\begin{figure}[tb]
    \begin{center}
          \includegraphics[height=3cm]{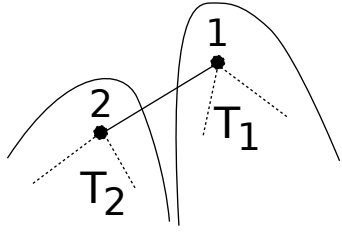}
     \end{center}
     \caption{A tree $T$ as a grafting of $T_2$ on $T_1$.}
     \label{FigGrafting}
\end{figure}

Let us denote, for a subset $X$ of $[r]$, $K_X = \sum_{i \in X} k_i$.
The weight of the tree $T_2 \bullet T_1$ obtained by grafting is given by the formula
\[\wt(T_2 \bullet T_1) = \wt(T_2) \wt(T_1) k_1 \left( K_{X(T_2)} -|X(T_2)| +1 \right), \]
so summing over all trees $T=T_2 \bullet T_1$, we obtain
\[ \sum_{\substack{T \text{ tree}, \\ X(T)=[r]}} \wt(T) =
\sum_{T_1,T_2} \wt(T_2) \wt(T_1) k_1 \left( K_{X(T_2)} -|X(T_2)| +1 \right). \]
The sum on the right-hand side runs over pairs of trees such that $X(T_1)$ contains $1$,
$X(T_2)$ contains $2$ and the sets $X(T_1)$ and $X(T_2)$ form a partition of $[r]$.
Splitting the sum according to the sets $X_h = X(T_h) \backslash \{h\}$ (for $h=1,2$), we obtain
\begin{multline}\label{EqHookInd} \!\!\!\!\!\! \sum_{\substack{T \text{ tree},\\ X(T)=[r]}} \!\!\!\!\!\! \wt(T)
= \!\!\!\!\!\! \sum_{\substack{X_1, X_2, \\ X_1 \sqcup X_2 = \{3,\dots,r\} }}
\!\!\!\!\!\!\!\!\!\!\!\! k_1 \left( k_2 + K_{X_2} -|X_2| \right) \\
\times\bigg( \!\!\!\!\!\! \sum_{\substack{T_1, \\ X(T_1)=\{1\} \sqcup X_1}} \!\!\!\!\!\! \wt(T_1) \bigg)
\bigg( \!\!\!\!\!\! \sum_{\substack{T_2, \\ X(T_2)=\{2\} \sqcup X_2}} \!\!\!\!\!\! \wt(T_2) \bigg).
\end{multline}
We now apply the induction hypothesis on the right-hand side to get, for $h=1,2$,
\[ \!\!\!\!\!\! \sum_{\substack{T_h, \\ X(T_h)=\{h\} \sqcup X_h}} \wt(T_h)
= k_h \bigg(\prod_{i \in X_h} k_i \bigg)  (k_h + K_{X_h} - 1)_{|X_h| - 1}. \]
Plugging this into~\eqref{EqHookInd}, we obtain
\[  \!\!\!\!\!\! \sum_{\substack{T \text{ tree}, \\ X(T) =[r]}} \wt(T)
= \bigg( \prod_{i=1}^r k_i\bigg) P(k_1,\dots,k_r) , \]
where
 \begin{equation}
  P(k_1,\dots,k_r):=\!\!\!\!\!\!\!\!\!\! \sum_{\substack{X_1, X_2, \\ X_1 \sqcup X_2 = \{3,\dots,r\} }} \!\!\!\!\!\!\!\!\! 
    k_1 (k_1 + K_{X_1} - 1)_{ |X_1| - 1}  (k_2 + K_{X_2} - 1)_{ |X_2| }.
    \label{EqDefP}
\end{equation}

In order to complete the inductive proof of our hook formula, we now prove that, for $r\geq 2$, $P(k_1,\dots,k_r)$ is equal to
\[Q(k_1,\dots,k_r) = (K-1)_{r-2}.\]
It is clear that both $\{ P(k_1,\dots,k_r)\}_{r \geq 2}$ and $\{ Q(k_1,\dots,k_r)\}_{r \geq 2}$ are families of multivariate polynomials, and
that, for each $r\geq 2$, $Q$ satisfies the following two properties:
\begin{itemize}
    \item As a polynomial in $k_1$, the constant term is
        \begin{equation}\label{EqCTP}
            Q(0,k_2,\dots,k_3) = (K_{\{2,\dots,r\}} -1)_{r-2} ;
        \end{equation}
    \item It satisfies the finite difference equation
        \begin{equation}\label{EqFiniteDiffP}
            \Delta_{k_1} Q(k_1,\dots,k_r) = \sum_{i = 3}^r
        Q(k_1+k_i,k_2,\dots,\widehat{k_i},\dots,k_r).
    \end{equation}
        Here $\Delta_{k_1}$ stands for the finite difference operator with respect
        to $k_1$, that is, $\Delta_{k_1} f(k_1)= f(k_1+1)-f(k_1)$, and the notation
        $\widehat{k_i}$ means that $k_i$ does {\em not} appear as an argument.
        \end{itemize}
These two properties completely determine the family of multivariate polynomials
$\{ Q(k_1,\dots,k_r)\}_{r \geq 2}$ (by immediate induction on $r$).
We now complete the proof that $P=Q$ by proving that the family $\{ P(k_1,\dots,k_r)\}_{r \geq 2}$
also has these two properties.

{\em Constant term}: If $X_1 \neq \emptyset$, then $(k_1 + K_{X_1} - 1)_{|X_1| - 1}$ is
 a polynomial in $k_1$, which implies that the summand corresponding to $X_1$
in Equation~\eqref{EqDefP} is a multiple of $k_1$.
Thus, the constant term of $P$ corresponds to the summand indexed by $X_1 = \emptyset$, which implies immediately that $P$ satisfies equation~\eqref{EqCTP}.

{\em Finite difference equation}: A simple computation gives
\begin{multline*}
    \Delta_{k_1} \big( k_1  (k_1 + K_{X_1} - 1)_{|X_1| - 1} \big)
    = (|X_1| k_1 + K_{X_1}) ( k_1 + K_{X_1} - 1)_{ |X_1| - 2} 
\end{multline*}
Therefore, from~\eqref{EqDefP} we obtain
\begin{multline} \label{EqDeltaP}
\Delta_{k_1} P(k_1,\dots,k_r) \\
= \!\!\!\!\!\!\!\! \sum_{\substack{X_1, X_2, \\ X_1 \sqcup X_2 = \{3,\dots,r\} }} \!\!\!\!\!\!\!\!
(|X_1| k_1 + K_{X_1})
 (k_1 + K_{X_1} - 1)_{ |X_1| - 2}
 (k_2 + K_{X_2} - 1)_{ |X_2| }.
\end{multline}

Also, directly from~\eqref{EqDefP}, we have
\begin{align*}
& \sum_{i=3}^r P(k_1+k_i,k_2,\dots,\widehat{k_i},\dots,k_r)\\
&= \sum_{i=3}^r\sum_{\substack{ Y_1, Y_2, \\ Y_1 \sqcup Y_2 = \{3,\dots,r\} \setminus \{i\} }} \!\!\!\!\!\!\!\!\!\!\!\!
(k_1 + k_i) 
 ( k_1 + k_i + K_{Y_1} - 1)_{ |Y_1| - 1}
 (k_2 + K_{Y_2} - 1)_{|Y_2|} \\
 &= \sum_{i=3}^r \sum_{\substack{X_1, X_2, \\ X_1 \sqcup X_2 = \{3,\dots,r\},\ i \in X_1}} \!\!\!\!\!\!\!\!\!\!\!\!
    (k_1 + k_i) 
    (k_1 + K_{X_1} - 1)_{|X_1| - 2}
    (k_2 + K_{X_2} - 1)_{|X_2|} \\
 &= \!\!\!\!\!\!\!\!\!\!\!\! \sum_{\substack{X_1, X_2, \\ X_1 \sqcup X_2 = \{3,\dots,r\}}} \!\!\!\!
  \bigg( \sum_{i \in X_1} (k_1 + k_i) \bigg) 
  (k_1 + K_{X_1} - 1)_{ |X_1| - 2}
  (k_2 + K_{X_2} - 1)_{|X_2|},
   \end{align*}
 where we have changed summation indices from the first equation above to the second by setting $X_1 = Y_1 \sqcup \{i\}$ and $X_2=Y_2$. Comparing this with~\eqref{EqDeltaP}
 implies immediately that $P$ satisfies equation~\eqref{EqFiniteDiffP}, which completes the proof that $P=Q$, and hence the first direct proof of our hook formula.
 
\section{Multivariate Lagrange inversion}
\label{SectProof2}

For the second direct proof of our hook formula (Theorem~\ref{ThmHookFormula}), we apply Lagrange inversion in many variables.
We again proceed by induction on $r$, with base case $r=1$, for which the theorem is trivially true.
Now consider an arbitrary (unordered increasing) tree $T$  of size $r > 1$.
The root vertex labelled $1$ has degree $j$ for some $j\geq 1$,
and the tree decomposes into $j$ sub-trees, whose vertex sets form a partition of $\{ 2,\ldots ,r\}$.
From this analysis we immediately obtain the following recurrence relationship
for the combinatorial sum on the left-hand side of the hook formula in Theorem~\ref{ThmHookFormula}:

\begin{equation}\label{CombTreeRec}
\sum_T \wt (T)=\sum_{j\geq 1}\frac{k_1^j}{j!} \sum_{\substack{ X_1 \sqcup\cdots\sqcup X_j \\ =\{ 2,\ldots ,r\} }}
\prod_{i=1}^j \left( K_{X_i}-\vert X_i\vert +1\right) \!\!\!\! \sum_{T_i: X(T_i)=X_i} \!\!\!\!\!\!\!\wt (T_i).
\end{equation}

We complete the proof by showing that the algebraic expression on the right-hand side of the hook formula in Theorem~\ref{ThmHookFormula} also satisfies this recurrence equation. To do so, we apply the following multivariate form of Lagrange's Implicit Function Theorem, as given in Goulden and Jackson~\cite{GJbook}, Theorem~1.2.9(1). 
\begin{theorem}\label{MVL}
 Suppose that $w_i=t_i\phi_i(\bfw )$, where $\phi_i$ is a formal power series with constant term $1$, for $i=1,\ldots ,r$, with $\bfw =(w_1,\ldots ,w_r)$. Then for integers $n_1,\ldots ,n_r$ and formal Laurent series $f$, we have

\begin{align*}
&[t_1^{n_1}\cdots t_r^{n_r}]f(\bfw ) \\
&=[\lambda_1^{n_1}\cdots \lambda_r^{n_r}]f(\bfla ) \phi_1(\bfla )^{n_1}\cdots \phi_r(\bfla )^{n_r}
\det \left(\delta_{ij}-\frac{\lambda_j}{\phi_i(\bfla )}\frac{\partial\phi_i(\bfla )}{\partial \lambda_j}\right)_{1\leq i,j\leq r},
\end{align*}
where $\bfla =(\lambda_1,\ldots ,\lambda_r)$.
\end{theorem}

Applying this form of Lagrange's Theorem, we obtain the following identity.
\begin{theorem}\label{MVLidentity}
For $r\geq 2$, we have
\[ k_1\cdots k_r (K-1)_{r-2}
=\sum_{j\geq 1}\frac{k_1^j}{j!} \sum_{ X_1\sqcup\cdots\sqcup X_j=\{ 2,\ldots ,r\} }
 \prod_{i=1}^j \bigg( \prod_{\ell\in X_i}k_{\ell} \bigg)  (K_{X_i}-1)_{\vert X_i\vert -1}.
\]
\end{theorem}

\begin{proof}
Consider $\phi_i(\bfw )=(1+w_1+\cdots +w_r)^{k_i}$, for $i=1, \ldots ,r$. Then we have
\begin{align*}
\det \left(\delta_{ij}-\frac{\lambda_j}{\phi_i(\bfla )}\frac{\partial\phi_i(\bfla )}{\partial \lambda_j}\right)
&= \det \left(\delta_{ij}-\frac{\lambda_j k_i}{1+\lambda_1+\cdots +\lambda_r} \right)\\
&=1-\frac{\sum_{i=1}^r \lambda_i k_i}{1+\sum_{i=1}^r \lambda_i},
\end{align*}
since $\det (I+M)= 1+\trace\ M$ when $\rank\ M \leq 1$.

We now calculate $[t_1\cdots t_r] w_1$ in two ways. First, directly from Theorem~\ref{MVL}, with $n_1=\cdots =n_r=1$, and $f(\bfw )=w_1$, we obtain
\begin{align*}
[t_1\cdots t_r]w_1&=[\lambda_1\cdots\lambda_r]\lambda_1 (1+\sum_{i=1}^r \lambda_i)^K
\left( 1-\frac{\sum_{i=1}^r \lambda_i k_i}{1+\sum_{i=1}^r \lambda_i}\right)\\
&= (r-1)!\binom{K}{r-1} -(K-k_1)(r-2)!\binom{K-1}{r-2}\\
&= k_1 (K-1)_{r-2}.
\end{align*}

Second, applying the functional equation $w_1=t_1\phi_1(\bfw )$, we obtain
\begin{align*}
[t_1\cdots t_r]w_1&= [t_1\cdots t_r]t_1(1+\sum_{i=1}^r w_i)^{k_1}\\
&= [t_2\cdots t_r]\sum_{j\geq 0}\frac{k_1^j}{j!}\left(\log (1+\sum_{i=1}^r w_i)\right)^j\\
&=\sum_{j\geq 1} \frac{k_1^j}{j!} \sum_{X_1\sqcup \cdots\sqcup X_j = \{ 2, \ldots , r \} }
\prod_{i=1}^j \left(  [\prod_{x\in X_i} t_x]\log (1+\sum_{i=1}^r w_i ) \right).
\end{align*}

But, for any $X\subseteq\{ 2,\ldots ,r\}$, with $\vert X\vert = m\geq 1$, Theorem~\ref{MVL} gives
\begin{align*}
&[\prod_{x\in X}t_x] \log (1+\sum_{i=1}^r w_i ) \\
&= [\prod_{x\in X}\lambda_x] \log (1+\sum_{i=1}^r \lambda_i )
(1+\sum_{i=1}^r \lambda_i )^{K_X}\left( 1-\frac{\sum_{i=1}^r \lambda_i k_i}{1+\sum_{i=1}^r \lambda_i}\right)\\
&= [\prod_{x\in X}\lambda_x] \log (1+\sum_{x\in X} \lambda_x )
(1+\sum_{x\in X}\lambda_x )^{K_X}\left( 1-\frac{\sum_{x\in X} \lambda_x k_x}{1+\sum_{x\in X} \lambda_x}\right)\\
&= m![z^m]\log(1+z) (1+z)^{K_X} - K_X (m-1)! [z^{m-1}]\log (1+z) (1+z)^{K_X -1}\\
&=(m-1)![z^{m-1}] \left\{ \frac{d}{dz} \left( \log (1+z) (1+z)^{K_X}  \right)   - \log (1+z) \frac{d}{dz} (1+z)^{K_X} \right\}\\
&= (m-1)![z^{m-1}]\frac{1}{1+z} (1+z)^{K_X}
= (m-1)!\binom{K_X -1}{m-1} = (K_X -1)_{m-1}.
\end{align*}
The result follows by equating the two expressions for $[t_1\cdots t_r] w_1$, and then multiplying by $k_2\cdots k_r$.
\end{proof}

It follows immediately from Theorem~\ref{MVLidentity} that the algebraic expression on the right-hand side of the hook formula in Theorem~\ref{ThmHookFormula} also satisfies recurrence equation~\eqref{CombTreeRec}, and this completes the second direct proof of our hook formula.

\section*{Acknowledgements}
We thank an anonymous referee for pointing out some references.

\bibliographystyle{abbrv}

\bibliography{1205}

\def\cprime{$'$}
\begin{thebibliography}{10}

\bibitem{BianeAsymptoticsCharacters}
P.~Biane.
\newblock Representations of symmetric groups and free probability.
\newblock {\em Adv. Math.}, 138(1):126--181, 1998.

\bibitem{Biane2003}
P.~Biane.
\newblock Characters of symmetric groups and free cumulants.
\newblock In {\em Asymptotic combinatorics with applications to mathematical
  physics (St. Petersburg, 2001)}, volume 1815 of {\em Lecture Notes in Math.},
  pages 185--200. Springer, Berlin, 2003.

\bibitem{BjoernerWachsQHook}
A.~Björner and M.~L. Wachs.
\newblock $q$-hook length formulas for forests.
\newblock {\em Journal of Combinatorial Theory, Series A}, 52(2):165--187,
  1989.

\bibitem{BorchardtCayley}
C.~Borchardt.
\newblock {\"U}ber eine der {I}nterpolation entsprechende {D}arstellung der
  {E}liminations-{R}esultante.
\newblock {\em Journal f{\"u}r die Reine und angewandte Mathematik},
  1860(57):111--121, 1860.

\bibitem{VicAdrienAlainEtMoi}
A.~Boussicault, V.~Féray, A.~Lascoux, and V.~Reiner.
\newblock Linear extension sums as valuations on cones.
\newblock {\em J. Alg. Comb.}, pages 1--38, 2012.
\newblock doi: 10.1007/s10801-011-0316-2.

\bibitem{BedardGoupilFactorizations}
F.~Bédard and A.~Goupil.
\newblock The poset of conjugacy classes and decomposition of products in
  symmetric group.
\newblock {\em Canad. Math. Bull.}, 35(2):152--160, 1992.

\bibitem{CayleyTrees}
A.~Cayley.
\newblock A theorem on trees.
\newblock {\em Quart. J. Math}, 23:376–378, 1889.

\bibitem{ChapotonArbustes}
F.~Chapoton.
\newblock Une opérade anticyclique sur les arbustes.
\newblock {\em Annales mathématiques Blaise Pascal}, 17(1):17--45, 2010.

\bibitem{ChapotonFlorentMarnitonsMoules}
F.~Chapoton, F.~Hivert, J.-C. Novelli, and J.-Y. Thibon.
\newblock An operational calculus for the mould operad.
\newblock {\em Int. Math. Res. Not.}, 2008, 2008.
\newblock 22 pp.

\bibitem{ChenEtAlHookFormulasCombi}
W.~Chen, O.~Gao, and P.~Guo.
\newblock On {H}an’s hook length formulas for trees.
\newblock {\em Elec. J. Comb.}, 18:P155, 2011.

\bibitem{DolegaFeraySniady2008}
M.~Do{\l}\k{e}ga, V.~F{\'e}ray, and P.~{\'S}niady.
\newblock Explicit combinatorial interpretation of {K}erov character
  polynomials as numbers of permutation factorizations.
\newblock {\em Adv. Math.}, 225(1):81--120, 2010.

\bibitem{DolegaFerayKerovJack}
M.~Do{\l}\k{e}ga and V.~Féray.
\newblock On {K}erov polynomials for {J}ack characters.
\newblock {arXiv:1201.1806}, 2012.

\bibitem{DuLiuGeneralizedHookSum}
R.~Du and F.~Liu.
\newblock $(k, m)$-{C}atalan numbers and hook length polynomials for plane
  trees.
\newblock {\em Eur. J. Comb.}, 28(4):1312--1321, 2007.

\bibitem{FaharatHigman1959}
H.~Farahat and G.~Higman.
\newblock The centres of symmetric group rings.
\newblock {\em Proc. Roy. Soc. (A)}, 250:212--221, 1959.

\bibitem{VicEtMoiHook}
V.~F{\'e}ray and V.~Reiner.
\newblock P-partitions revisited.
\newblock {\em Journal of Commutative Algebra}, 2012.
\newblock to appear.

\bibitem{HookLengthFormula1954}
J.~S. Frame, G.~d.~B. Robinson, and R.~M. Thrall.
\newblock The hook graphs of the symmetric group.
\newblock {\em Canadian Journal of Mathematics}, 6:316--324, 1954.

\bibitem{GJbook}
I.~P. Goulden and D.~M. Jackson.
\newblock {\em Combinatorial Enumeration}.
\newblock Wiley-Interscience Series in Discrete Mathematics. J. Wiley and Sons,
  New York, 1983 (Dover reprint 2004).

\bibitem{GouldenJackson1992}
I.~P. Goulden and D.~M. Jackson.
\newblock The combinatorial relationship between trees, cacti and certain
  connection coefficients for the symmetric group.
\newblock {\em European J. Combin.}, 13(5):357--365, 1992.

\bibitem{GouldenJacksonMacdonaldSym}
I.~P. Goulden and D.~M. Jackson.
\newblock Symmetrical functions and {M}acdonald's result for top connexion
  coefficients in the symmetrical group.
\newblock {\em Journal of Algebra}, 166(2):364--378, 1994.

\bibitem{HanHookLengthAutomatic}
G.~Han.
\newblock Discovering hook length formulas by an expansion technique.
\newblock {\em Elec. J. Comb.}, 15(R133):1, 2008.

\bibitem{HivertNovelliThibonHookEquaDiff}
F.~Hivert, J.-C. Novelli, and J.-Y. Thibon.
\newblock Trees, functional equations, and combinatorial {H}opf algebras.
\newblock {\em Eur. J. Comb.}, 29(7):1682--1695, 2008.

\bibitem{IvanovOlshanski2002}
V.~Ivanov and G.~Olshanski.
\newblock Kerov's central limit theorem for the {P}lancherel measure on {Y}oung
  diagrams.
\newblock In {\em Symmetric functions 2001: surveys of developments and
  perspectives}, volume~74 of {\em NATO Sci. Ser. II Math. Phys. Chem.}, pages
  93--151. Kluwer Acad. Publ., Dordrecht, 2002.

\bibitem{KerovOlshanski1994}
S.~V. Kerov and G.~I. Olshanski.
\newblock Polynomial functions on the set of {Y}oung diagrams.
\newblock {\em C. R. Acad. Sci. Paris}, Serie. I(319):121--126, 1994.

\bibitem{KnuthArtProg3}
D.~Knuth.
\newblock {\em The Art of Computer Programming, Vol. 3: Sorting and Searching}.
\newblock Addison-Wesley, 1973.

\bibitem{McDo}
I.~Macdonald.
\newblock {\em Symmetric functions and Hall polynomials}.
\newblock Oxford Univ. Press, second edition, 1995.

\bibitem{PostnikovHook}
A.~Postnikov.
\newblock Permutohedra, associahedra, and beyond.
\newblock {\em International Mathematics Research Notices}, 2009(6):1026--1106,
  2009.

\bibitem{ProctorDynkin}
R.~Proctor.
\newblock Dynkin diagram classification of $\lambda$-minuscule {B}ruhat
  lattices and of $d$-complete posets.
\newblock {\em Journal of Algebraic Combinatorics}, 9(1):61--94, 1999.

\bibitem{ProctorMinuscule}
R.~Proctor.
\newblock Minuscule elements of weyl groups, the numbers game, and-complete
  posets.
\newblock {\em Journal of Algebra}, 213(1):272--303, 1999.

\bibitem{RattanSniady2008}
A.~Rattan and P.~{\'S}niady.
\newblock Upper bound on the characters of the symmetric groups for balanced
  {Y}oung diagrams and a generalized {F}robenius formula.
\newblock {\em Adv. Math.}, 218(3):673--695, 2008.

\bibitem{Sagan}
B.~Sagan.
\newblock {\em The symmetric group : Representations, Combinatorial Algorithms,
  and Symmetric functions}.
\newblock Springer, New York, second edition, 2001.

\bibitem{SaganProbabilisticHookFormulas}
B.~Sagan.
\newblock Probabilistic proofs of hook length formulas involving trees.
\newblock {\em S{\'e}minaire Lotharingien de Combinatoire}, 61:B61Ab, 2009.

\bibitem{StanleyEC1}
R.~Stanley.
\newblock {\em Enumerative Combinatorics, Vol. I}.
\newblock Wadsworth \& Brooks/Cole, 1986.

\bibitem{SunZhangHookFormulaMaryTrees}
Y.~Sun and H.~Zhang.
\newblock Two kinds of hook length formulas for complete m-ary trees.
\newblock {\em Discrete Mathematics}, 309(8):2584--2588, 2009.

\bibitem{YangGeneralizationHanHookFormula}
L.~Yang.
\newblock Generalizations of {H}an's hook length identities.
\newblock arXiv preprint 0805.0109, 2008.

\end{thebibliography}
\end{document}